\documentclass[12pt, anon]{l4dc2024}
\usepackage{graphicx} 
\usepackage[utf8]{inputenc} 
\usepackage[T1]{fontenc}    
\PassOptionsToPackage{hyphens}{url}\usepackage{hyperref}
\usepackage{booktabs}       
\usepackage{amsfonts}       
\usepackage{amsmath}
\usepackage{nicefrac}       
\usepackage{microtype}      
\usepackage{xcolor}
\usepackage{times}
\usepackage{breakcites}
\allowdisplaybreaks
\newtheorem{assumption}{Assumption}

\newcommand{\mb}{\mathbb}
\newcommand{\mc}{\mathcal}

\usepackage[british]{babel}
\usepackage{hyperref}  
\usepackage{enumerate}
\usepackage{enumitem}

\usepackage{algorithm}
\usepackage{algorithmic}

\allowdisplaybreaks

 \title{A Large Deviations Perspective on Policy Gradient Algorithms}
\author{%
\Name{Wouter Jongeneel} \Email{wouter.jongeneel@epfl.ch}\\
 \addr EPFL
 \AND
 \Name{Daniel Kuhn} \Email{daniel.kuhn@epfl.ch}\\
 \addr EPFL
   \AND
  \Name{Mengmeng Li} \Email{mengmeng.li@epfl.ch}\\
 \addr EPFL
}

\begin{document}

\maketitle
\begin{abstract}
    Motivated by policy gradient methods in the context of reinforcement learning, we identify a large deviation rate function for the iterates generated by stochastic gradient descent for possibly non-convex objectives satisfying a Polyak-Łojasiewicz condition. Leveraging the contraction principle from large deviations theory, we illustrate the potential of this result by showing how convergence properties of policy gradient with a softmax parametrization and an entropy regularized objective can be naturally extended to a wide spectrum of other policy parametrizations.
\end{abstract}

\section{Introduction and related work.}

Policy gradient methods are at the core of several reinforcement learning (RL) algorithms.~\textit{e.g.}, see~\cite{sutton1999policy,ref:degris2012model,agarwal2021theory}. As such, a wide adoption of these algorithms is aided by precisely understanding their performance. To that end, in scenarios where policy gradients can only be evaluated stochastically, it is crucial to understand the underlying distribution governing the evolution of policy iterates. Despite the popularity of stochastic policy gradient methods for over 20 years, global convergence has been only understood recently~\citep{fazel2018global,zhang2020global,bhandari2019global,agarwal2021theory}. 
Although global convergence proofs are actively developed for different policy gradient algorithm templates~\citep{bhandari2019global,agarwal2021theory,cen2022fast}, two aspects of the current convergence analysis would benefit from further study: 1) global convergence guarantees are often stated in terms of \textit{expected} suboptimality; and 2) the choice of policy parametrizations impacts convergence behavior~\citep{ref:mei2020escaping}.
In light of these observations, this paper aims to derive sharp convergence rates in \textit{probability} and to provide a unifying approach towards understanding the effect of different policy parametrizations. 
As we build upon a rich history of work, we succinctly comment on: (i) \textit{high probability analysis} in stochastic optimization and RL; and (ii) \textit{policy gradient} methods in RL.


(i)~Despite stochastic gradient descent (SGD) being over 70 years old~\citep{ref:RobbinsMonro51}, there has been a surge of interest in terms of \textit{high probability analysis} for (SGD): for strongly-convex objectives~\citep{ref:harvey2019tight}; non-convex objectives satisfying a Polyak-Łojasiewicz (PL) condition~\citep{madden2020high}; and non-convex objectives~\citep{ref:ghadimi2013stochastic,liu2023high} with Lipschitz gradients. A motivation for studying tight convergence bounds in terms of probabilities over convergence in expectation is due to the fact that the iterates generated by SGD can be brittle and \textit{large deviations} from the expected value can occur,~\textit{e.g.}, see~\cite{ref:gower2020variance} and references therein. Hence, a high-probability bound is attractive since it reassures the practitioner that certain behavior occurs at least with a close-to-$1$ probability for a single realization of iterates.

(ii)~\textit{Policy gradient} methods can be interpreted as gradient descent applied to the policy optimization problem in RL~\cite[Ch.~13]{sutton2018reinforcement}, see in particular~\cite[p.~337]{sutton2018reinforcement} for historical remarks. Therefore, when the state space is large~\citep{ref:bottou2010large} and due to necessary approximation, one oftentimes resorts to stochastic variants, whose convergence guarantees are often stated in terms of expectations~\citep{lan2023policy}, with the exception of~\citep{ding2021beyond}, see also~\cite[\S~2.2]{madden2020high}. The high-probability convergence rate provided in~\cite{ding2021beyond} can be seen as a loosened bound of the \textit{large deviation rate} we aim to derive. The pursuit of a tight high-probability concentration bound for the iterations produced by stochastic policy gradient methods holds significant potential for bolstering the practical implementation and interpretability of policy gradient methods, thereby improving the overall applicability of RL.
Studying the concentration behavior of SGD iterates through the lens of large deviations is pioneered by 
\cite{bajovic2022large}. It is worth mentioning that, however, their analysis exploits strong convexity of the optimization objective. Then, a first step towards understanding more precisely the concentration behavior of stochastic policy gradient iterates would be to study its entropy-regularized variant, which already faces the difficulties of non-convexity and non-uniformity of the PL constant.


As the theory of \textit{large deviations} is the key tool in this paper, we briefly introduce it already at this stage. Let $X_t\in \mathbb{R}^d$ be the $t^{\mathrm{th}}$ iterate of some stochastic algorithm aimed at driving $X_t$ to $x^{\star}$ for $t\to +\infty$, then, many aforementioned works provide statistical guarantees of the form $\mathbb{P}(\|X_t-x^{\star}\|_2\leq \varepsilon)\geq 1 - \beta$ $\forall t\geq T$ for an appropriately chosen triple $(T,\varepsilon,\beta)$. 
Although often impressive pieces of work, there are some remarks to be made. First, it is often not clear if these high-probability bounds are \textit{tight}. 
Secondly, the discrepancy is often measured by a sufficiently simple function, like the $\ell_2$-norm in this case, however, in practice one might be interested in vastly different sublevel sets,~\textit{e.g.}, sublevel sets capturing harmful events, the ones that are of interest in safe RL. Third, the guarantees are oftentimes states for particular policy parametrizations, a unifying framework that generalizes performance analysis across different policy parametrizations remains largely unexplored.
The theory of \textit{large deviations} is precisely powerful when one aims to address these aforementioned points. 
Let us clarify terminology and provide intuition, yet, while skipping some details. A sequence of finite (probability) measures $(\mu_t)_t$ on $\mathbb{R}^d$ is said to satisfy a \textit{large deviation principle} (LDP) with \textit{rate function} $I:\mathbb{R}^d\to [0,+\infty]$ when for any Borel set $\Theta\subseteq \mathbb{R}^d$ 
\begin{align}
\label{equ:LDP:def}
    -\mathring{r}:=-\inf_{\theta\in \mathring{\Theta}}I(\theta) \leq \liminf_{t\to +\infty} \frac{1}{t}\log \mu_t({\Theta})\leq \limsup_{t\to +\infty} \frac{1}{t}\log \mu_t({\Theta})\leq -\inf_{\theta\in \overline{\Theta}}I(\theta)=: -\overline{r},
\end{align}
where $\mathring{\Theta}$ denotes the interior of~$\Theta$, and $\overline{\Theta}$ denotes the closure of~$\Theta.$
Now, identifying $(\mu_t)_t$ with a sequence of random variables $(X_t)_t$, we have after rearranging~\eqref{equ:LDP:def} that 
$\mathrm{exp}(-\mathring{r}t+o(t)) \leq \mathbb{P}(X_t\in \Theta) \leq \mathrm{exp}(-\overline{r}t+o(t)).$
As such, an LDP captures---possibly tight,~\textit{e.g.}, under mild topological assumptions---convergence rates.
Indeed, regarding applicability, once the rate function is identified, the corresponding inaccuracy rate can be obtained for \textit{any} new given region of interest. For a detailed exposition of large deviations, we refer to an overview paper by~\cite{rs2008special} and the book by~\cite{dembo2011large}. As alluded to above, we will use the theory of large deviations to study stochastic policy gradient iterates, that is, to study the iterates of a non-convex optimization algorithm. Although large deviation tools appeared in seminal work on non-convex optimization~\cite{ref:ghadimi2013stochastic}, that work does not provide an LDP. What is more, large deviations are intimately connected to rare events, whose analysis is of interest to RL~\citep{frank2008reinforcement}. Also, efficient adaptive sampling techniques become available thanks to careful deployment of large deviation theory~\citep{dupuis2004importance}, an approach that is recently attracting interest in the context of SGD~\citep{lahire2023importance}.

\paragraph{Contributions.} 
\begin{enumerate}[label=(\roman*)]
    \item We find, with high probability, an {lower} bound on the rate function for iterates generated by softmax policy gradient with an entropy regularized objective~\eqref{expr:PG:maxent:update}, see Theorem~\ref{thm:LDP:bound}. We also recover results similar to~\cite{madden2020high}, yet, via large deviation theory, see Lemma~\ref{lem:LDP:up}.
    \item \label{item:2} We demonstrate the wide applicability of having established a large deviation rate by leveraging the contraction principle to establish large deviation rates for a wide range of tabular policy parametrizations beyond the softmax parametrization. 
    
    \item Effectively, we establish a LDP upper bound for SGD under a PL condition, which is of independent interest.   
\end{enumerate}
\vspace{-5pt}
In Section~\ref{sec:problem:statement} we describe the RL setting under consideration whereas in Section~\ref{sec:prelim} we derive a preliminary exponential bound on the convergence of the value function in probability. Then, in Section~\ref{sec:LDP} we provide our main result: a LDP upper bound for the stochastic policy gradient iterates. At last, we show in Section~\ref{sec:ramifications} that this LDP upper bound can be lifted from a softmax parametrization to a whole family of parametrizations, which directly leads to exponential convergence rates for existing and new policy parametrizations, 
with high probability.

\section{Problem statement.}
\label{sec:problem:statement}
\paragraph{Markov decision process (MDP).} We consider a finite MDP given by a six-tuple $\left(\mathcal{S}, \mathcal{A}, P, c, \gamma, \rho\right)$ consisting of a finite state space $\mathcal{S}=\{1,\ldots,S\}$, a finite action space $\mathcal{A}=\{1,\ldots,A\}$, a transition kernel $P: \mathcal{S} \times \mathcal{A} \rightarrow \mc P(\mc S)$, a cost-per-stage function $c: \mathcal{S} \times \mathcal{A} \rightarrow \mathbb{R}$, assumed to be bounded\footnote{One can set the cost of a particular stage to be $\infty$ but that results in an infeasible/unbounded optimization problem.}, a discount factor $\gamma\in(0,1)$,
and an initial distribution $\rho\in\mc P(\mathcal{S})$. Here, we use
$\mc P(\mc S):=\{p \in \mathbb{R}^{|\mathcal{S}|} : \ \sum_{i=1}^{|\mathcal{S}|} p_i=1, p \geq 0\}$ to denote the probability simplex over~$\mathcal{S}$. 
Throughout the rest of the paper we restrict attention to stationary policies, which are described by a stochastic kernel $\pi\in\Pi:=\mc P(\mc A)^{|\mathcal{S}|}$.
The value function $V^\pi: \mc P(\mc S) \rightarrow \mathbb{R}$ associated with $\pi$ is defined through
\begin{align}
\label{equ:V:pi:s}
& V^\pi(\rho):=\mathbb{E}\big[\textstyle\sum_{k=0}^{\infty} \gamma^kc\left(s_k, a_k\right)| s_0\sim \rho, a_k \sim \pi\left(\cdot | s_k\right), s_{k+1} \sim P(\cdot | s_k, a_k)\big],
\end{align}
and the main objective is to minimize $V^\pi(\rho)$ across all $\pi\in\Pi.$

\paragraph{Entropy-regularized reinforcement learning (RL).} 
It is sometimes useful to work with the modified objective where a regularizer is added, that is,
$
V^\pi_\tau(\rho):=V^\pi(\rho)+\tau \cdot \mb H(\rho, \pi)
$~\citep{ref:neu2017unified,mei2020global}, where $\tau>0$ and
$$
\mb H(\rho, \pi):={\mathbb{E}}\big[\textstyle\sum_{k=0}^{\infty}-\gamma^k \log \pi\left(a_k | s_k\right)| s_0 \sim \rho, a_k \sim \pi\left(\cdot | s_k\right), s_{k+1} \sim P\left(\cdot | s_k, a_k\right)\big].
$$
An optimal policy $\pi^\star\in\Pi$ is defined through the condition $V_\tau^{\pi^\star}(
\rho)\le V_\tau^{\pi}(\rho)$ for all $\pi\in\Pi.$

\paragraph{Softmax policy gradient for entropy-regularized objective.} The softmax transform $\pi_\theta(\cdot | s):=$ $\operatorname{softmax}(\theta(s, \cdot))$ of any function $\theta: \mathcal{S} \times \mathcal{A} \rightarrow$ $\mathbb{R}$ is defined through
\begin{equation}
\label{equ:softmax}
\pi_\theta(a | s)=\frac{\mathrm{exp}\left({\theta(s, a)}\right)}{\sum_{a'\in \mathcal{A}} \mathrm{exp}\left({\theta\left(s, a'\right)}\right)} \quad \forall a \in \mathcal{A}.
\end{equation}
In the remainder, we frequently think of $\theta$ as a policy parameter in $\mathbb{R}^{d=|\mathcal{S}||\mathcal{A}|}$ and we use $V^\theta_\tau(\rho)$ for $V^{\pi_\theta}_{\tau}(\rho)$. By~\citep[Thm. 3]{nachum2017bridging} there exists $\theta^\star\in\mb R^d$ such that $\pi_{\theta^\star}$ is the optimal policy of the entropy-regularized MDP.
By the policy gradient theorem~\citep{sutton1999policy}, the policy gradient $g(\theta_t):=\partial V^{\theta}_\tau(\rho)/\partial \theta |_{\theta=\theta_t} $ is given by
$\partial V^{\theta}_\tau(\rho)/\partial \theta(s, a)=1/(1-\gamma) \cdot d_\rho^{\theta}(s) \cdot \pi_\theta(a | s) \cdot A^{\theta}_\tau(s, a)$
, where $d_\rho^{\theta}$ denotes the discounted state-visitation frequency measure, and $A^{\theta}_\tau$ is the advantage function associated with the regularized objective, see also~\cite[Lem.~1]{mei2020global}. In practice, however, computing the advantage function is computationally expensive. 
It is therefore convenient to work with estimates of the exact gradient. Throughout the rest of the paper we  assume unbiased stochastic sample access to $\partial_{\theta} V^{\pi_\theta}_\tau(\rho)|_{\theta=\theta'}$, denoted as $\tilde g(\theta')$.
Stochastic softmax policy gradient with entropy-regularized objective thus suggests to update policy parameters via
\vspace{-5pt}
\begin{equation}\label{expr:PG:maxent:update}
\theta_{t+1} \leftarrow \theta_t-
\eta_t \cdot \tilde g(\theta_t).
\end{equation}
\vspace{-10pt}
\begin{assumption}[Warm start with sufficient exploration]
\label{ass:init}
    The initial state distribution satisfies $\allowbreak\min_{s\in\mc S} \rho(s) >0$. In addition, the initial iterate $\theta_1$ is chosen around a $\Delta$-neighborhood of $\theta^\star$ in the sense that {$\min_{\theta^\star\in\Theta^\star}\|\theta_1-\theta^\star\|_2\le \Delta$ where~$\Theta^\star$ is the set of all optimal solutions.}
\end{assumption}
 The above assumption 
 simplifies the exposition and such an initial $\theta_1$ can be obtained by a stochastic policy gradient method such as~\citep[Alg. 3.2]{ding2021beyond}. 
 We additionally define $Z_t := g(\theta_t)-\tilde g(\theta_t).$ As in~\citep{bajovic2022large} we make the following assumption. 
\begin{assumption}[Gradient estimation uncertainty]
\label{ass:noise:ZM}
    The stochastic process $(Z_t)_t$ satisfies: $(i)$ $Z_t$ depends on the past only through $\theta_t$; $(ii)$ $\mathbb{E}[Z_t|\theta_t=\theta]=0$ for any $\theta$ and $(iii)$ the distribution of $Z_t$ is independent of the iterate index $t\in \mathbb{N}$.
\end{assumption}
By Assumption~\ref{ass:noise:ZM}~$(i)$, the process $(\theta_t)_{t\ge0}$ generated by~\eqref{expr:PG:maxent:update} is a Markov chain.
We use $\Lambda(\lambda;\theta) = \log \mathbb{E}[\mathrm{exp}(\langle \lambda,Z_t \rangle)\mid \theta_t=\theta]$ to denote the conditional \textit{log-moment generating function} (LMGF) of $Z_t\in \mathbb{R}^d$. We also define $M(\nu;\theta) = \mathbb{E}[\mathrm{exp}(\nu \|Z_t\|_2^2)\mid \theta_t=\theta]$ as the \textit{conditional moment-generating function} (MGF) of $\|Z_t\|_2^2$. Now, we impose the following technical assumption to ensure that the tail probabilities of the disturbances $(Z_t)_t$ decay sufficiently fast. 

\begin{assumption}[Sub-Gaussian process]
\label{ass:subG}
    All elements of the sequence $(Z_t)_t$ follow a $\sigma$-sub-Gaussian distribution for some $\sigma>0$,~\textit{i.e.}, 
    \begin{equation}
\label{equ:LMGF:sG}
    \Lambda(\lambda;\theta)\leq \tfrac{1}{2}\sigma^2 \|\lambda\|_2^2   \quad \forall \lambda, \theta\in \mathbb{R}^d. 
\end{equation}
\end{assumption}
A Monte-Carlo gradient estimation method that satisfies the above assumptions is proposed by~\cite{ding2021beyond}. The method operates by generating trajectories of the MDP with policy parameter $\theta$ and then estimate the advantage function based on these trajectories. 
\vspace{-5pt}
\begin{lemma}[$L_1$-smoothness of the value function {\citep[Lem.~7 \& 14]{mei2020global}}]
\label{lemma:smooth:value}
There exists a constant $L_1>0$ such that
    \begin{equation}
    \label{equ:L1:upper}
    |V^{\pi_\theta}_\tau(\rho)-V^{\pi_{\theta'}}_\tau(\rho) - \langle g(\theta'),\theta-\theta'\rangle | \leq \tfrac{1}{2}{L_1} \|\theta-\theta'\|_2^2 \quad \forall \theta,\theta'\in \mb R^d.
\end{equation}
\end{lemma}
For the dependency of $L_1$ on the four-tuple $(c,\gamma,\tau,|\mathcal{A}|)$ see~\cite[Lem.~7 \& 14]{mei2020global}.

\section{Preliminaries.}
\label{sec:prelim}

\begin{lemma}[Non-uniform Polyak-Łojasiewicz condition]
\label{lem:nu:PL}
If Assumption~\ref{ass:init} holds, then $\left\|g(\theta)\right\|_2^2 \geq \mu(\theta)\big(V_\tau^\theta(\rho)-V_\tau^{\theta^{\star}}(\rho)\big),$ where $\mu(\theta)=2 \tau|\mathcal{S}|^{-1} \min _s \rho(s) \min _{s, a} \pi_\theta(a | s)^2\|d_\rho^{\pi_{\theta^{\star}}}/\rho\|_{\infty}^{-1}.$
Moreover, if Assumption~\ref{ass:noise:ZM} and~\ref{ass:subG} hold and $\eta_t\le \eta /(t+\sqrt{3 \sigma^2/(2 \epsilon \Delta)})$, for iterates generated by~\eqref{expr:PG:maxent:update}, we have 
$ \mu \!= \inf_{t\ge 1} \mu(\theta_t)\!>\!0$ with probability at least $1-\epsilon/6$.
\end{lemma}
\begin{proof}
By \citep[Lem.~15]{mei2020global}, we have $\left\|g(\theta)\right\|_2^2 \geq \mu(\theta)\big(V_\tau^\theta(\rho)-V_\tau^{\theta^{\star}}(\rho)\big),$ where $\mu(\theta)=2 \tau|\mathcal{S}|^{-1} \min _s \rho(s) \min _{s, a} \pi_\theta(a | s)^2\|d_\rho^{\pi_{\theta^{\star}}}/\rho\|_{\infty}^{-1}.$ According to~\cite[Lem.~6.4]{ding2021beyond} we further have $\inf_{t\ge 1}\pi_{\theta_t}(a|s)>0$ with probability $1-\epsilon/6$, which concludes the proof.
\end{proof}
Lemma~\ref{lem:nu:PL} states that along a trajectory of~\eqref{expr:PG:maxent:update}, with high probability, we have that $\mu(\theta_t)>0$ for all $t$. In what follows 
we will exploit that $ \mu \!= \inf_{t\ge 1} \mu(\theta_t)$ is strictly positive with high probability.
Next, we derive elementary recursive inequalities and provide the first main result. 
If~$\eta_t$ satisfies $0<\eta_t\le 1/L_1$ 
for all $t=1,\ldots,T$, then~\eqref{expr:PG:maxent:update} implies that
\begin{align}
    V^{\theta_{t+1}}_\tau(\rho)
  &= V^{\theta_{t}-\eta_t g(\theta_t)+\eta_t Z_t}_\tau(\rho) \nonumber
  \\&\le V^{\theta_{t}}_\tau(\rho) - \eta_t \langle g(\theta_t), g(\theta_t)-Z_t\rangle + \tfrac{1}{2} {\eta_t^2 L_1}\|g(\theta_t)-Z_t\|_2^2 \nonumber 
  \\&\le  V^{\theta_{t}}_\tau(\rho) - \eta_t\| g(\theta_t)\|_2^2 +\eta_t \langle g(\theta_t),Z_t\rangle + {\eta_t^2 L_1}\big(\|g(\theta_t)\|_2^2+\|Z_t\|_2^2\big) \nonumber
  \\&\le V^{\theta_{t}}_\tau(\rho) +\mu \left( {\eta_t^2 L_1}-\eta_t \right)\big(V^{\theta_{t}}_\tau(\rho) -V_\tau^{\theta^{\star}}(\rho)\big) +\eta_t \langle g(\theta_t),Z_t\rangle + {\eta_t^2 L_1}\|Z_t\|_2^2, \nonumber
\end{align}
where the first inequality follows from the $L_1$-smoothness of the value function, the second inequality exploits that $\|x-y\|_2^2\leq 2\|x\|_2^2+2\|y\|_2^2$ for all $x,y\in \mathbb{R}^d$, and the last inequality holds because $\eta_t\le 1/L_1$ and the PL condition from Lemma~\ref{lem:nu:PL}.
Subtracting $V_\tau^{\theta^{\star}}(\rho)$ from both sides yields
\begin{align}\label{recursion:start}
    &V^{\theta_{t+1}}_\tau(\rho) - V_\tau^{\theta^{\star}}(\rho) 
    \le (1-\mu\eta_t +\mu \eta_t^2 L_1)(V^{\theta_{t}}_\tau(\rho) -V_\tau^{\theta^{\star}}(\rho)) +\eta_t \langle g(\theta_t),Z_t\rangle + {\eta_t^2 L_1}\|Z_t\|_2^2. 
\end{align}
Also, observe that since\footnote{To show that $\mu\le L_1$, one exploits that $L_1$-smoothness implies Lipschitz continuity of the gradient and recalls that the PL condition implies a quadratic growth condition~\citep[App.~A]{ref:karimi2016linear}.} $\eta_t\leq 1/L_1$, $\mu\leq L_1$  and $\eta_t,L_1>0$, we have $\mu \eta_t (\eta_t L_1-1)\in (-1,0]$ and thus $(1-\mu\eta_t +\mu \eta_t^2 L_1)\in (0,1]$ for $t=1,\dots, T$. 
Next, by substituting the points $\theta'=\theta-1/L_1\cdot g(\theta)$ and $\theta$ into~\eqref{equ:L1:upper}, we obtain
$V^{\theta-1/L_1\cdot g(\theta)}_\tau(\rho) \le V^{\theta}_\tau(\rho) - \tfrac{1}{2}L_1^{-1} \|g(\theta)\|_2^2$.
Using the optimality of $\theta^\star$, we may then conclude that $V_\tau^{\theta^{\star}}(\rho)\le V^{\theta-1/L_1\cdot g(\theta)}_\tau(\rho) \le V^{\theta}_\tau(\rho) - \tfrac{1}{2}L_1^{-1} \|g(\theta)\|_2^2$ which implies that
\begin{align}\label{gradient:smooth}
    \|g(\theta)\|_2^2&\le 2L_1(V^{\theta}_\tau(\rho)-V^{\theta^\star}_\tau(\rho)) \quad \forall \theta\in \mathbb{R}^d.
\end{align}
\vspace{-15pt}
\begin{lemma}[Exponential upper bound]
\label{lem:LDP:up}
Suppose Assumptions~\ref{ass:init},~\ref{ass:noise:ZM}, and \ref{ass:subG} hold, and choose $T>1$, $\epsilon\in (0,1)$. Then, there is a universal constant $C>0$ such that the following event holds with probability at least $1-\epsilon/6$. Let $C_{\mathsf{M}}=(\sigma \sqrt{|\mathcal{S}| |\mathcal{A}|} C)^2$ where $\sigma$ is as in~\eqref{equ:LMGF:sG},  and set $\eta>0$ such that $(\mu\eta-1) > \sigma^2/ C_{\mathsf{M}}$. Let $\eta_t=\eta/(t+t_0+1)$ for all $t=1,\ldots,T$ 
    and choose $t_0$ and $K$ such that
    \begin{align}
       \nonumber t_0 \geq & \max\left\{\frac{\eta^2 L_1}{(\mu\eta-1) - B_0 C_0 \eta^2}-1,L_1 \eta - 2,\sqrt{\frac{3 \sigma^2}{2 \epsilon \Delta}}-1 \right\},\\
       \label{equ:K} K \geq & \max_{t=1,\dots,T} \left\{B_0^{-1}, (t_0+1)\big(V^{\theta_{1}}_\tau(\rho) - V_\tau^{\theta^{\star}}(\rho) \big), \frac{2c_t C_{\mathsf{M}}}{1-(a_t+B_0C_0 b_t^2)} \right\}, 
    \end{align}
    where $B_0=1/(2 \eta^2 L_1 C_{\mathsf{M}})$, $C_0=2L_1 \sigma^2$, 
    \begin{align}
    \label{equ:abc}
a_t =\frac{t+t_0+1}{t+t_0}\left(1-\mu\eta_t +\mu \eta_t^2 L_1\right) , \
b_t  =\frac{\eta}{\sqrt{t+t_0}}, \ \text{and} \
c_t  =\frac{\eta^2 L_1}{t+t_0+1}.
\end{align}
Then, for any $\delta\geq 0$ and $t=1,\dots, T$ we have that 
        $
        \mb P\left(\big(V^{\theta_{t+1}}_\tau(\rho) - V_\tau^{\theta^{\star}}(\rho)\big) \ge \delta\right) \le e^{1-(t+t_0+1)\delta/K}.
        $
\end{lemma}
The proof is inspired by~\citep{bajovic2022large}, with the difference being that they work with strongly convex objective functions. Unfortunately, $V^\theta_\tau(\rho)$ fails to be convex in~$\theta.$
\begin{proof}
 Define an auxiliary process $Y_{t} = (t+t_0)(V^{\theta_{t}}_\tau(\rho) - V_\tau^{\theta^{\star}}(\rho) ).$ Then, the recursion~\eqref{recursion:start} implies that
    $Y_{t+1} \le a_t Y_t + b_t\sqrt{t+t_0}\langle g(\theta_t),Z_t\rangle + c_t\|Z_t\|^2,$
where $a_t$, $b_t$ and $c_t$ are as in~\eqref{equ:abc}.
Defining the MGFs associated with $Y_t$ as $\Phi_t(\nu)  =\mathbb{E}\left[\exp \left(\nu Y_t\right)\right]$ and $\Phi_{t+1 | t}\left(\nu ; \theta_t\right)  =\mathbb{E}\left[\exp \left(\nu Y_{t+1}\right) \mid \theta_t\right]$ for $\nu \in \mathbb{R}$, the recursion~\eqref{recursion:start} further implies that \vspace{-5pt}
\begin{align*}
 \Phi_{t+1 | t}\left(\nu ; \theta_t\right)
& \!\leq \exp \left(a_t \nu Y_t\right)\mathbb{E}\left[\exp \left(2 b_t \nu \sqrt{t+t_0}\langle g(\theta_t),Z_t\rangle\right) \mid \theta_t\right]^{1/2}
\mathbb{E}\!\left[\exp \big(2 c_t \nu\left\|Z_t\right\|_2^2\big) \mid \theta_t\right]^{1/2} \\
& \leq \exp \left(a_t \nu Y_t\right) \exp \left(\sigma^2 b_t^2 \nu^2 (t+t_0) \|g(\theta_t)\|_2^2\right)^{1/2}\mathbb{E}\left[\exp \big(2 c_t \nu\left\|Z_t\right\|_2^2\big) \mid \theta_t\right]^{1 / 2}\\ 
& \leq \exp \left(a_t \nu Y_t\right) \exp \left(2L_1 \sigma^2 b_t^2 \nu^2 Y_t\right)^{1/2}\mathbb{E}\left[\exp \big(2 c_t \nu\left\|Z_t\right\|_2^2\big) \mid \theta_t\right]^{1 / 2}, 
\end{align*}
where the first inequality follows from \textit{H\"older's inequality}, the second inequality follows from~\eqref{equ:LMGF:sG} and the last inequality follows from~\eqref{gradient:smooth}. Recall that $d=\mathrm{dim}(Z_t)$ and define $B_0=1/(2\eta^2 L_1(\sigma \sqrt{d} C)^2)$.
Then, there is a constant $C>0$ such that for all $\nu \in 
[0, B_0]$, we have by monotonicity of $c_t$, H\"older's inequality for $p=1/(\nu 2 \eta^2 L_1 (\sigma \sqrt{d} C)^2)\geq 1$ and~\citep[Lem.~2]{ref:jin2019short}
that $M(2c_t\nu;\theta)\leq \mathrm{exp}(2c_t \nu (\sigma \sqrt{d} C)^2)$. Recall that $(\sigma \sqrt{d} C)^2=C_{\mathsf{M}}$ because $d=|\mc S||\mc A|$.  
Moreover, as $\mathrm{exp}(x)\geq \mathrm{exp}(x)^{1/2}$ for $x\geq 0$, 
we can simplify the above inequality to 
\begin{equation}
    \label{equ:phi:rec}
    \begin{aligned}
    \Phi_{t+1 | t}\left(\nu ; \theta_t\right) &\leq \exp \left(a_t \nu Y_t\right) \exp \left(2L_1 \sigma^2 b_t^2 \nu^2 Y_t\right)\mathrm{exp}(2c_t \nu C_{\mathsf{M}})\\
    & 
    = \mathrm{exp}\left(\nu\left(a_t+\nu C_0 b_t^2 \right)Y_t \right)\mathrm{exp}(2c_t \nu C_{\mathsf{M}})\quad \forall\nu\leq B_0
\end{aligned}
\end{equation}
for $C_0 = 2L_1 \sigma^2$. Taking expectations on both sides of~\eqref{equ:phi:rec} yields
\begin{equation}
    \label{equ:phi:rec:2}
    \Phi_{t+1}(\nu) \leq \Phi_t\left(\nu(a_t+B_0 C_0 b_t^2 )\right)\mathrm{exp}(2c_t \nu C_{\mathsf{M}}) \quad\forall \nu\le B_0. 
\end{equation}
We aim to show that $\Phi_{t+1}(\nu)\leq \mathrm{exp}(\nu K)$ for $t=1,\dots, T$ by studying the magnitude of $(a_t+B_0 C_0 b_t^2 )$ and using induction.
First, it can be shown that $a_t<1$ for all $t\geq 1$. To do so, rewrite $a_t$ as
\begin{equation}
\label{equ:at}
    a_t = 1 - \frac{\mu\eta -1}{t+t_0}\left(1- \frac{\mu \eta^2 L_1}{(\mu \eta - 1)(t+t_0+1)} \right),
\end{equation}
where we exploit the identity $1+1/s=(s+1)/s$ for $s\neq 0$. We have that $(\mu \eta - 1)>0$. Now since $t_0\geq L_1 \eta - 2$ by assumption, we have $\eta_t\le 1/L_1$. Recall as well the fact that $\mu \leq L_1$, which leads to $\eta_t \leq 1/\mu$ and thus $\mu \eta_t \leq 1$, that is, $\mu \eta/(t+t_0+1)\leq 1$. Exactly this implies that $(\mu\eta -1)/(t+t_0)\in (0,1]$. This covers the first non-trivial fraction of~\eqref{equ:at}. Now for the second part, we have the implication that 
\begin{equation*}
    t_0 +1 \geq \frac{\mu\eta^2 L_1}{\mu \eta -1}\implies \frac{\mu\eta^2 L_1}{(\mu \eta - 1)(t+t_0+1)}\in (0,1], 
\end{equation*}
from where we can conclude that $a_t \in [0,1)$ for any finite $t\geq 1$. However, we will need $(a_t+B_0 C_0 b_t^2)<1$. As~$b_t$ decays with~$t_0$, this can be achieved by selecting a sufficiently large $t_0$. 
Explicitly, let $\mu \eta = 1 + \varepsilon$ for some $\varepsilon>0$.
By the definition of~$b_t$ and $a_t$ as in~\eqref{equ:at} we then obtain
\begin{equation}
\label{equ:at:2}
    a_t+B_0C_0b_t^2 = 1 - \frac{\varepsilon}{t+t_0}\left(1-\frac{B_0C_0\eta^2}{\varepsilon}- \frac{\mu\eta^2 L_1}{\varepsilon(t+t_0+1)} \right).
\end{equation}
By assumption, we have $\mu\eta-1>\sigma^2/C_{\mathsf{M}}.$ This implies that $\varepsilon>B_0 C_0 \eta^2.$
It then follows that $a_t+B_0C_0 b_t^2\in [0,1)$ for any finite $t\geq 1$ and 
  $  t_0 +1\geq \mu\eta^2 L_1/(\varepsilon - B_0 C_0 \eta^2).$
Recall also that $t_0\geq L_1 \eta - 2$ by assumption. Thus, we have $\eta_t\le 1/L_1$.
Assumption~\ref{ass:init} allows then for the following inductive procedure. Let $\nu\in [0,1/K]$ for $K$ as in~\eqref{equ:K}
then, for $t=1$, we have $\Phi_1(\nu)\leq \mathrm{exp}(\nu K)$ since $K\geq t_0 \left(V^{\theta_{1}}_\tau(\rho) - V_\tau^{\theta^{\star}}(\rho) \right)$. Now suppose $\Phi_t(\nu)\leq \mathrm{exp}(\nu K)$ for some arbitrary $t\geq 1$, then, we have by the inductive assumption that 
    $\Phi_{t+1}\leq \mathrm{exp}(2 c_t C_{\mathsf{M}}\nu)\mathrm{exp}(\nu(a_t+B_0 C_0 b_t^2)K). $  
However, since for $t=1,\dots, T$ we also have that
   $ K\geq 2c_t C_{\mathsf{M}}/(1-(a_t+B_0C_0 b_t^2)),$
which is well-defined since $(a_t+B_0 C_0 b_t^2)<1$, by our selection of $t_0$ and $\eta$, it follows that $\Phi_{t+1}(\nu)\leq \mathrm{exp}(\nu K)$ for $t=1,\dots, T$.
By~\cite[Claim~A.7]{ref:harvey2019tight}, if a random variable $X\in \mathbb{R}$ satisfies $\mathbb{E}[\mathrm{exp}(\lambda X)]\leq C_1 \mathrm{exp}(\lambda C_2)$ for all $\lambda \leq 1/C_2$ and some universal constants $C_1$ and $C_2$, then $\mathbb{P}(X \geq C_2 \log(1/\delta))\leq C_1 e \delta$. 
Hence, after applying~\cite[Claim~A.7] {ref:harvey2019tight} one finds that 
    $\mathbb{P}(Y_t\geq C_3)\leq \mathrm{exp}(1-C_3/K), $
 for any choice of $C_3\geq 0$ and $t=1,\dots, T$. The proof is concluded by substituting the expression for $Y_t$ into the previous inequality and a union bound. 
\end{proof}
\noindent We note that by the expressions for $(a_t,b_t,c_t)$ and~\eqref{equ:at:2} it follows that $\lim_{t\to+\infty}2c_t C_{\mathsf{M}}/(1-(a_t+B_0C_0 b_t^2))$ is bounded, that is, one can freely select $T$, which we will exploit later in Theorem~\ref{lem:cumulant:LMGF}.



\section{Large deviations.}
\label{sec:LDP}

To provide our main result, we need to impose another assumption on the random variables $(Z_t)_t$. 
\begin{assumption}[Conditional LMGF regularity]
\label{ass:CLMGF:reg}
    There is a $L_{\Lambda}\geq 0$ such that
    \begin{equation*}
        |\Lambda(\lambda;\theta_1)-\Lambda(\lambda;\theta_2)|\leq L_{\Lambda}\|\lambda\|_2^2 \|\theta_1 - \theta_2\|_2\quad \forall (\lambda,\theta_1,\theta_2)\in \mathbb{R}^d\times \mathbb{R}^d \times \mathbb{R}^d.
    \end{equation*}
\end{assumption}
We refer to~\cite[p.~5]{bajovic2022large} for a discussion of this assumption. For instance, when the conditional distribution of $Z_t$ given $\theta_t=\theta$ is Gaussian, and its covariance matrix is Lipschitz continuous in $\theta$, then Assumption~\ref{ass:CLMGF:reg} holds. In particular, if the covariance is independent of $\theta$, this trivially holds true. Next, we define the MGF of $\theta_t$ as $\mathbb{R}^d \ni \lambda \mapsto\Gamma_t(\lambda):=\mathbb{E}\left[\mathrm{exp}(\langle \lambda, \theta_t - \theta^{\star}\rangle) \right]$ and similarly the LMGF as $\mathbb{R}^d\ni \lambda \mapsto \log \Gamma_t (\lambda)$. Now we are equipped to provide our main technical result, leading directly to a large deviation principle (LDP) upper bound. Note that $\Gamma_t(\lambda)$ depends on the optimal solution $\theta^\star$ which is fixed and unknown.

\begin{theorem}[Limiting LMGF]
\label{lem:cumulant:LMGF}
    Suppose that Assumptions~\ref{ass:init},~\ref{ass:noise:ZM},
    ~\ref{ass:subG} and~\ref{ass:CLMGF:reg} hold and set $\eta_t=\eta/(t+t_0+1)$ as in Lemma~\ref{lem:LDP:up}. Then, conditioning on the event that $ \mu \!= \inf_{t\ge 1} \mu(\theta_t)\!>\!0$, we have that
    \begin{equation}\label{limit:LMGF:upper}
    \limsup_{t\to +\infty} \frac{1}{t}\log \Gamma_{t}(t\lambda) \leq r(\lambda)  +\int_0^1 \Lambda\left(\eta Q D(x) Q^{\top} \lambda ; \theta^{\star}\right) \mathrm{d} x=:\Psi(\lambda), 
    \end{equation}
    where $r(\lambda)=O(\|\lambda\|^3_2)$ with its expression presented in the proof below, and $Q$ and $D(x)$ are such that
    $H(\theta^{\star})=QDQ^\top$, $QQ^\top=I_d$,
    $D=\mathrm{diag}(\rho_1,\dots,\rho_n)$ being the diagonalization of $H(\theta^{\star})$,
    and $D(x)=\mathrm{diag}(x^{\eta \rho_1 -1},\ldots,x^{\eta \rho_n -1})$.
\end{theorem}

To prove Theorem~\ref{lem:cumulant:LMGF}, we will largely follow the proof strategy of~\cite[Lem.~6]{bajovic2022large}, yet, with technical deviations regarding indexing of the sequences. In addition, we cannot appeal to the strong convexity exploited in~\citep{bajovic2022large}.

\begin{proof}
We provide a brief overview of the proof. Step 0 constructs a deterministic sequence as a basis for deriving recursions of MGF. Step 1 defines a ball around the optimal parameter and split the cases depending on whether $\theta_\ell$ resides within the ball. Step 2-4 study individual cases leading to their respective recursive relation. Finally, Step 5 summarizes the previously derived recursions, takes the limit and concludes. 
\textit{Step 0, sequence.} Fix some vector $\lambda\in \mathbb{R}^d$, integer $t\geq 1$ and define the sequence
of vectors $\zeta_{\ell} = B_{t,\ell} \zeta_t$ via $B_{t,\ell}=(I-\eta_{\ell}H(\theta^{\star}))\cdot ... \cdot (I-\eta_{t-1} H(\theta^{\star}))$ and $\zeta_{t} = t\lambda$, with $1\leq \ell < t$. As $\eta_t\leq 1/L_1$ we have that $\|I-\eta_t H(\theta^{\star})\|_2\leq 1 - \eta_t \lambda_{\mathrm{min}}(H(\theta^{\star}))$. Hence, using~\cite[Lem.~2]{bajovic2022large} and exploiting ${\eta \lambda_{\mathrm{min}}(H(\theta^{\star}))}>1$,~\textit{i.e.}, recall $\eta \mu > 1$ from Lemma~\ref{lem:LDP:up},  one can show that 
\begin{equation*}
    \|\zeta_{\ell}\|_2\leq t \left(\frac{\ell+t_0+1}{t+t_0+1} \right)^{\eta \lambda_{\mathrm{min}}(H(\theta^{\star}))}\|\lambda\|_2 \leq (\ell+t_0+1)\|\lambda\|_2.
\end{equation*}
\textit{Step 1, conditional MGF.}
Let $\mu_{\ell}$ and $\nu_{\ell}$ be Borel measures induced by $\theta_{\ell}$ and $\|\theta_{\ell}-\theta^{\star}\|_2$, respectively. Then, one readily shows that 
   $ \Gamma_{\ell+1}(\zeta_{\ell}) = \textstyle\int_{\mathbb{R}^d} \Gamma_{\ell+1|\ell}(\zeta_{\ell};\theta)\mu_{\ell}(\mathrm{d}\theta).$ 
Next, define the sets $B_{\ell}(\theta^{\star},\delta) = \{\theta_{\ell}:\|\theta_{\ell}-\theta^{\star}\|_2\leq \delta\}$.
Now we can construct the decomposition $\Gamma_{\ell+1}(\zeta_{\ell})=\Gamma_{\ell+1|B_{\ell}(\theta^{\star},\delta)}(\zeta_{\ell})+\Gamma_{\ell+1|B^c_{\ell}(\theta^{\star},\delta)}(\zeta_{\ell})$, which we study separately.
\textit{Step 2, $\theta\in B_{\ell}(\theta^{\star},\delta)$.}
We have that 
    $\Gamma_{\ell+1|\ell}(\zeta_{\ell};\theta) =  \mathrm{exp}\left(\Lambda (\eta_{\ell}\zeta_{\ell};\theta) + \langle \zeta_{\ell} , \theta- \eta_{\ell}g(\theta) - \theta^{\star} \rangle \right).$
Let $H(\theta)$ denote the Hessian of $V^{\pi_{\theta}}_\tau(\rho)$ at $\theta$ and define the residual term $h(\theta) = g(\theta)-H(\theta^{\star})(\theta-\theta^{\star})$. 
Recall Step 0, Assumption~\ref{ass:CLMGF:reg} and define the largest residual term $\Bar{h}(\delta) = \sup_{\theta\in B_{\ell}(\theta^{\star},\delta)}\|h(\theta)\|_2$, it follows that 
\begin{align*}
     \Gamma_{\ell+1|\ell}(\zeta_{\ell};\theta) \leq & \mathrm{exp}\left(\Lambda (\eta_{\ell}\zeta_{\ell};\theta^{\star}) + L_{\Lambda}\eta_{\ell}^2 \|\zeta_{\ell}\|_2^2 \delta + \eta_{\ell}\|\zeta_{\ell}\|_2 \bar{h}(\delta) + \langle \zeta_{\ell-1},\theta-\theta^{\star}\rangle\right)\\
 \leq & \mathrm{exp}\left(\Lambda (\eta_{\ell}\zeta_{\ell};\theta^{\star}) + r_0(\lambda,\delta)\right)\mathrm{exp}(\langle \zeta_{\ell-1},\theta-\theta^{\star}\rangle)
\end{align*}
for $r_0(\lambda,\delta) = 4L_{\Lambda}\eta^2 \|\lambda\|_2^2 \delta + 2\eta\|\lambda\|_2 \bar{h}(\delta)$.
Hence, integrating, we find that 
\begin{equation}
    \label{equ:bound:inner}
    \Gamma_{\ell+1|B_{\ell}(\theta^{\star},\delta)}(\zeta_{\ell}) \leq  \mathrm{exp}\left(\Lambda (\eta_{\ell}\zeta_{\ell};\theta^{\star}) + r_0(\lambda,\delta)\right)\Gamma_{\ell}(\zeta_{\ell-1}).
\end{equation}
\textit{Step 3, contraction.} 
We construct another recursive formula. Start from
\vspace{-5pt}
\begin{align*}
    \|\theta_{t+1}-\theta^{\star}\|_2 &= \|\theta_t - \eta_t g(\theta_t) + \eta_t Z_t - \theta^{\star}\|_2\leq \|\theta_t - \eta_t g(\theta_t)  - \theta^{\star}\|_2 + \eta_t \|Z_t\|_2,
\end{align*}
and expand 
$    \|\theta_t - \eta_t g(\theta_t)  - \theta^{\star}\|^2_2 = \|\theta_t - \theta^{\star}\|_2^2- 2\eta_t \langle \theta_t-\theta^{\star},g(\theta_t) \rangle + \eta_t^2 \|g(\theta_t)\|_2^2.$
Due to optimality of $\theta^{\star}$ and $L_1$-smoothness of $V^{\theta}_{\tau}$ we have $\|g(\theta_t)\|_2^2 \leq L_1^2 \|\theta_t-\theta^{\star}\|_2^2$ and
\begin{align*}
    - \langle \theta_t-\theta^{\star},g(\theta_t) \rangle &\leq \tfrac{1}{2}L_1 \|\theta_t - \theta^{\star}\|_2^2 + (V^{\theta_t}_{\tau}(\rho)-V^{\theta^{\star}}_{\tau}(\rho))\\
    &\leq \tfrac{1}{2}L_1 \|\theta_t - \theta^{\star}\|_2^2 + \frac{1}{2\mu(\theta_t)}\|g(\theta_t)\|_2^2
    \leq \tfrac{1}{2}L_1 \|\theta_t - \theta^{\star}\|_2^2 + \frac{L_1^2}{2\mu(\theta_t)}\|\theta_t - \theta^{\star}\|_2^2.
\end{align*}
As such, one can set $\gamma_t:=(1+\eta_t L_1 + \eta_t L_1^2/M\mu(\theta_t)+\eta_t^2 L_1^2)^{1/2}\leq (3 + L_1/\mu)^{1/2}=:\Bar{\gamma}$ such that 
\begin{align}
\label{equ:theta:rec}
    \|\theta_{t+1}-\theta^{\star}\|_2
    &\leq \gamma_t\|\theta_t  - \theta^{\star}\|_2 + \eta_t \|Z_t\|_2 \quad t=1,\dots, T. 
\end{align}
\textit{Step 4, $\theta\in B^c_{\ell}(\theta^{\star},\delta)$.} 
Using Step 3 and Assumption~\ref{ass:subG} one can show that \newline
   $ \allowbreak\Gamma_{\ell+1|\ell}(\zeta_{\ell};\theta) \leq \mathrm{exp}(\tfrac{1}{2}\sigma^2\eta^2 \|\lambda\|_2^2 )\mathrm{exp}(\Bar{\gamma}(\ell+t_0+1)\|\lambda\|_2 \|\theta-\theta^{\star}\|_2). $
Hence, we have
\begin{align*}
    \Gamma_{\ell+1|B^c_{\ell}(\theta^{\star},\delta)}(\zeta_{\ell}) \leq \mathrm{exp}\left(\tfrac{1}{2}\sigma^2\eta^2 \|\lambda\|_2^2 \right) \int_{\tau \geq \delta}\mathrm{exp}(\Bar{\gamma}(\ell+t_0+1)\|\lambda\|_2 \tau)\nu_{\ell}(\mathrm{d}\tau). 
\end{align*}
Exploiting Lemma~\ref{lem:LDP:up},~\textit{i.e.} $\nu_{\ell}$ being induced by $\mb P\left(\|\theta_{\ell}-\theta^{\star}\|_2 \ge \delta\right) \le e^{1-(\ell+t_0)L_1\delta^2/(2K)}$, we know there is a sufficiently large $\delta(\lambda)=O(\Bar{\gamma}\|\lambda\|_2 K/L_1)$, denoted as $\bar{\delta}(\lambda)$ such that  
\begin{align}
\label{equ:bound:outer}
    \Gamma_{\ell+1|B^c_{\ell}(\theta^{\star},\delta)}(\zeta_{\ell}) \leq \Bar{C}\mathrm{exp}\left(\tfrac{1}{2}\sigma^2\eta^2 \|\lambda\|_2^2 \right) 
\end{align}
for some constant $\Bar{C}$~\textit{cf.}~\cite[Proof of Lem.~6]{bajovic2022large}.
\textit{Step 5, limiting recursion.} Combining~\eqref{equ:bound:inner} and~\eqref{equ:bound:outer} then yields
  $ \Gamma_{\ell+1}(\zeta_{\ell}) \leq \Bar{C}\mathrm{exp}\left(\tfrac{1}{2}\sigma^2\eta^2 \|\lambda\|_2^2 \right) + \mathrm{exp}\left(\Lambda (\eta_{\ell}\zeta_{\ell};\theta^{\star}) + r(\lambda)\right)\Gamma_{\ell}(\zeta_{\ell-1}),$
where $r(\lambda)=r_0(\lambda,\bar{\delta}(\lambda))$. Continuing with the recursion provides us with
\begin{align*}
    \Gamma_{t+1}(t\lambda) \leq &\mathrm{exp}\left(\textstyle\sum^t_{\ell=1} \Lambda(\eta_{\ell}\zeta_{\ell};\theta^{\star}) + r(\lambda)\right)\Gamma_1(\zeta_1) \\&+\Bar{C}\mathrm{exp}\left(\tfrac{1}{2}\sigma^2\eta^2 \|\lambda\|_2^2 \right)\textstyle\sum^t_{\ell=1}\mathrm{exp}\left(\textstyle\sum^t_{j=\ell} \Lambda(\eta_{j}\zeta_{j};\theta^{\star}) + r(\lambda)\right)
\end{align*}
Note that $\Gamma_1(\zeta_1)$ is finite by Assumption~\ref{ass:init}.
Then we immediately obtain that
\begin{equation*}
    \limsup_{t\to +\infty} t^{-1}\log \Gamma_{t+1}(t\lambda) \leq r(\lambda) + \limsup_{t\to +\infty}t^{-1}\textstyle\sum^{t}_{\ell=1}\Lambda(\eta_{\ell}\zeta_{\ell};\theta^{\star}). 
\end{equation*}
To complete the proof, we appeal to~\citep[Lemma C.3]{bajovic2022large}
and find that 
\begin{align*}
    \lim _{t \rightarrow+\infty} t^{-1} \textstyle\sum_{\ell =1}^{t} \Lambda(\eta_{\ell}\zeta_{\ell};\theta^{\star})=\int_0^1 \Lambda\left(\eta Q D(x) Q^{\top} \lambda ; \theta^{\star}\right) \mathrm{d} x.
\end{align*}
\vspace{-10pt}
\end{proof}
\noindent To provide our main result, we recall that the \textit{Legendre-Fenchel transform} of a function $\Psi:\mathbb{R}^d\to \mathbb{R}$ is defined by $I(\theta')=\sup_{\lambda\in \mathbb{R}^d}\langle \theta',\lambda\rangle - \Psi(\lambda)$ for all $\theta'$. 

{
The Theorem below is a strict generalization of Lemma~\ref{lem:LDP:up}. Specifically, Theorem~\ref{thm:LDP:bound} characterizes the concentration rate applicable to any Borel set. 
}
\begin{theorem}[LDP upper bound]
\label{thm:LDP:bound}
Suppose that Assumption~\ref{ass:init},~\ref{ass:noise:ZM},
\ref{ass:subG} and~\ref{ass:CLMGF:reg} hold and set $\eta_t=\eta/(t+t_0+1)$ as in Lemma~\ref{lem:LDP:up}. Define $\Psi$ as in Theorem~\ref{lem:cumulant:LMGF}. Then, with probability at least $1-\epsilon/6$ we have that the sequence $(\theta_t)_t$ satisfies a LDP upper bound with a rate function $I$. The function $I$ is the Legendre-Fenchel transformation of $\Psi$. That is, for any Borel set $\Theta \subseteq \mathbb{R}^{|\mathcal{S}||\mathcal{A}|}$, we have that 
\begin{equation}
\label{equ:thm}
    \limsup_{t\to +\infty}\frac{1}{t}\log \mathbb{P}(\theta_t \in \Theta) \leq - \inf_{\theta'+\theta^{\star}\in \overline{\Theta}} I(\theta').
\end{equation}
\end{theorem}
\begin{proof}
    Directly from Theorem~\ref{lem:cumulant:LMGF} and the \textit{G\"artner-Ellis theorem}~\citep{gartner1977large,ellis1984large}.
\end{proof}
Note, one can also bring the offset term $\theta^{\star}$ into the left-hand-side of~\eqref{equ:thm},~\textit{i.e.}, consider $\mathbb{P}(\theta_t-\theta^{\star} \in \Theta)$. Indeed, strictly speaking, the LDP upper bound is derived for the sequence $(\theta_t-\theta^{\star})_t$.


\section{Ramifications} 
\label{sec:ramifications}


One could argue that the results (bounds) from above contribute towards a better understanding of the softmax transformation, the effect of the regularization parameter $\tau$, the effect of initialization and the analysis of restrictive policy parametrizations, all active topics of research~\cite{ref:hennes2020neural,ref:mei2020escaping,li2021softmax,mei2020global}. 
However, more interesting is that Theorem~\ref{thm:LDP:bound} also allows for a simple proof of exponential decay, with high probability, for a whole family of policy parametrizations different from softmax.  

The structure of $\Psi$ (non-negativity and finiteness) immediately reveals that its Fenchel conjugate (Legendre-Fenchel transform) $I$ satisfies $I(0)=0$, better yet, $I(\theta')>0$ for all $\theta'\neq 0$. Hence, $I$ can be identified as a \textit{good} rate function~\cite[p.~4]{dembo2011large}. Moreover, for \textit{any} Borel set $\Theta$ such that $\theta^{\star}\not\in \overline{\Theta}$ we have, with high probability, that the probability $\mathbb{P}(\theta_t\in \Theta)$ decays exponentially fast. Not only that, we can provide a bound on the convergence rate for any of those sets,~\textit{i.e.}, let $r:=\inf_{\theta'+\theta^{\star}\in \overline{\Theta}} I(\theta')$, then, $\mathbb{P}(\theta_t \in \Theta) \leq \mathrm{exp}(-rt+o(t))$. 
To lift this observation to any continuous transformation of the softmax parametrization we need the following principle from the theory of large deviations. To avoid pathological examples, we assume all our sets to be topological Hausdorff spaces. 

\begin{theorem}[Contraction principle~{\cite[\S~4.2.1]{dembo2011large}}]
\label{thm:contraction}
Consider $U\subseteq \mb R^d$ and $W\subseteq\mb R^q, $ and $f:U\to W$ a continuous map. 
Define $I':W\to\mb [0,+\infty]$ through
    $I'(w)=\inf_{u\in U,w=f(u)} I(u)$ for all $w \in W,$
where $I:U\to\mb [0,+\infty]$ is a good rate function. Then, $I'$ is also a good rate function on~$W$.
\end{theorem}
Although we studied a softmax policy, the \textit{contraction principle}, originally used by~\cite{donsker1976asymptotic}, allows us to extend our LDP results to any continuous transformation $f$ of $\theta_t-\theta^{\star}$,~\textit{e.g.}, one recovers the \textit{escort transformation}~\citep{ref:mei2020escaping} by transforming $\theta_t$ component-wise by the map $\mathbb{R}\ni w\mapsto p\log|w|$ for $p\geq 1$. To be explicit, whereas the softmax policy is given by~\eqref{equ:softmax}, the escort parametrization is given by
$
\pi_\theta(a | s)={|{\theta(s, a)}|^p}/{\sum_{a'\in \mathcal{A}} |{\theta\left(s, a'\right)|^p}}$.
Hence, we directly infer that the exponential decay rate extends from the softmax- to the escort parametrization. Indeed, $w\mapsto p\log|w|$ is not a smooth function, yet, for $p\geq 2$ the escort policy parametrization is differentiable, which is precisely the setting for which an exponential decay rate under entropic regularization is provided by~\cite[Thm.~4]{ref:mei2020escaping}.
To formally summarize this observation, by combining Theorem~\ref{thm:LDP:bound} and Theorem~\ref{thm:contraction} we get the following.
\begin{corollary}[Going beyond softmax]
\label{cor:beyond}
Let $(\theta_t-\theta^{\star})_t\subset \mathbb{R}^{|\mathcal{S}||\mathcal{A}|}$ be a sequence corresponding to softmax policy gradient and let the sequence $(\omega_t-\omega^{\star})_t\subset \mathbb{R}^q$ be defined via a continuous map $f:\mathbb{R}^{|\mathcal{S}||\mathcal{A}|}\to \mathbb{R}^q$ through $\omega_t-\omega^{\star}:=f(\theta_t-\theta^{\star})$ $\forall t$, then, with probability at least $1-\epsilon/6$, 
\begin{equation}
\label{equ:cor}
    \limsup_{t\to +\infty}\frac{1}{t}\log \mathbb{P}(\omega_t \in \Omega) \leq - \inf_{\omega'+\omega^{\star}\in \overline{\Omega}} I'(\omega'),
\end{equation}
for any Borel set $\Omega\subseteq \mathbb{R}^q$ and $I'$ as in Theorem~\ref{thm:contraction}. 
\end{corollary}
Similar to the escort transformation, one can study \textit{spherical}- and \textit{Taylor} softmax~\citep{ref:BrebissonV15}. However, Corollary~\ref{cor:beyond} is ought to be most interesting to apply to prospective policy parametrizations leading to~\eqref{equ:cor} being non-trivial. Indeed, we emphasize that one must verify that the solution to the non-linear optimization problem~\eqref{equ:cor} is non-trivial, there is no free lunch. Moreover, the gradient estimation errors needs to remain well-behaved after the transformation, which may not be true for an arbitrary continuous mapping~$f$. We leave it to future work to explore the exact characterizations of the scope of Corollary~\ref{cor:beyond}. 

To remove the $(1-\epsilon/6)$-high probability statements throughout, future work aims at studying the multiphase algorithm from~\citep{ding2021beyond} in the context of large deviations, this requires being able to study a controlled sequence $(Z_t)_t$, which is a non-trivial extension. We also remark that in the context of SGD for an objective satisfying a PL condition, the LDP upper bound always holds, simply since $\mu>0$, \textit{cf.}~Lemma~\ref{lem:nu:PL}.
\bibliography{names.bib}

\end{document}